\documentclass[12pt]{amsart}
\usepackage{graphics,wrapfig, graphicx, mathrsfs,xcolor}
\usepackage{mathtools}
\usepackage{amssymb,amscd,graphicx,xcolor,subfig,tikz}
\usepackage{graphics}
\usepackage{indentfirst}
\usepackage{bm, enumerate,xcolor}
\usepackage[dvips]{epsfig}
\usepackage{latexsym}
\usepackage{float}

\usepackage{amsmath}
\usepackage{amssymb}
\usepackage[applemac]{inputenc}
\usepackage[T1]{fontenc}
\newtheorem{lemma}{Lemma}[section]
\newtheorem{theorem}{Theorem}[section]

\newtheorem{definition}{Definition}[section]
\newtheorem{remark}{Remark}[section]

\numberwithin{equation}{section} \numberwithin{theorem}{section}
\numberwithin{example}{section} \numberwithin{remark}{section}
\numberwithin{figure}{section} \numberwithin{algorithm}{section}
\mathtoolsset{showonlyrefs}

\begin{document}
\title[Density estimates]{Density estimates for Ginzburg-Landau energies with degenerate double-well potentials}
\author{Ovidiu Savin}
\address{Department of Mathematics, Columbia University, New York, NY 10027}\email{savin@math.columbia.edu}
\author{Chilin Zhang}
\address{School of Mathematical Sciences, Fudan University, Shanghai 200433, China}\email{zhangchilin@fudan.edu.cn}

\begin{abstract}
    We consider a class of Allen-Cahn equations associated with Ginzburg-Landau energies involving degenerate double-well potentials that vanish of order $m$ at the minima
    \begin{equation}
        J(v,\Omega)=\int_{\Omega}\Big\{|\nabla v|^{p}+(1-v^{2})^{m}\Big\}dx,\quad 1<p<m,
    \end{equation}
   and establish density estimates for the level sets of nontrivial minimizers $|v| \le 1$. This extends a result of Dipierro-Farina-Valdinoci  where the density estimates for such degenerate potentials were obtained for a bounded range of $m$'s. The original estimates for the classical case $p=m=2$ were established by Caffarelli-C\'ordoba.
\end{abstract}

\maketitle

\section{Introduction}
The Ginzburg-Landau energy was developed from the theory of Van der Waals (see \cite{Row79}) by Landau, Ginzburg and Pitaevski\u i in \cite{GP58,Landau37,Landau67} in order to describe the phenomenon of phase transitions in thermodynamics (see also \cite{AC72,CH58}). It consists of studying a global minimizer $u:\mathbb{R}^{n}\to[-1,1]$ of the energy
\begin{equation}\label{eq. classical GL}
    J(v,\Omega)=\int_{\Omega}\Big\{|\nabla v|^{2}+(1-v^{2})^{2}\Big\}dx,
\end{equation}
in which $v$ represents the mean field of the spin of the particles in a lattice. Its Euler-Lagrange equation is the so-called Allen-Cahn equation:
\begin{equation}\label{eq. classical AC}
    \Delta u=2(u^3-u).
\end{equation}
Apart from the trivial minimizers $u\equiv\pm1$, the more complicated and interesting question is to study minimizers or critical points representing phase transitions, i.e. a solution that can be sufficiently close to both $1$ and $-1$ (two steady states), but with a phase field region $|u|\leq1-\epsilon$ in between.


In this paper, we consider the following generalization of the classical Ginzburg-Landau energy:
\begin{itemize}
    \item The Dirichlet energy is generalized to $|\nabla v|^{p}$ for $1<p<\infty$ due to more complicated mutual forces between particles;
    \item The potential energy is a generic double-well potential $W(v)$ with degenerate minima at $\pm 1$, for example: $W(v) \sim (1-v^{2})^{m}$.
\end{itemize}
Then, the Ginzburg-Landau energy becomes
\begin{equation}\label{eq. GL energy pm intro}
    J(v,\Omega)=\int_{\Omega}\Big\{|\nabla v|^{p}+W(v)\Big\}dx
\end{equation}
with a typical infinitesimal potential energy
\begin{equation}
    W(v)=(1-v^{2})^{m}.
\end{equation}
The Euler-Lagrange equation of a critical point for $J(u, \Omega)$ is
\begin{equation}
    p\Delta_{p}u=p\cdot\mathrm{div}(|\nabla u|^{p-2}\nabla u)=W'(u).
\end{equation}

The heteroclinical solution is the monotone one-dimensional solution (unique up to translations) that connects the stable phases $- 1$ and $+1$ as $x$ ranges from $-\infty$ to $ \infty$. The rate of decay of this solution to the limits $\pm 1$ depends on the values of $m$ and $p$. Precisely, if $m <p$ then the limits are achieved outside a finite interval, while if $m=p$ the rate of decay is exponential. On the other hand if $m>p$, the decay rate is polynomial. This last case can be viewed as a more degenerate situation since the convergence to the stable phases occurs at a much slower rate than exponential. Equivalently, the width of the phase field region $|u|\le 1-\varepsilon$ of the heteroclinical solution is $O(\varepsilon^{-\alpha(m)})$ for some positive exponent $\alpha(m)$ which increases with $m$. 

 It is well known that phase transitions modeled by minimizers $u\in W^{1,p}(B_R,[-1,1])$ of $J$ defined in a large ball $B_R \subset \mathbb R^n$ are closely connected to minimal surfaces. More precisely, the rescaling of the transition region $\frac 1R \{|u| \le 1-\varepsilon\}$ of $u$ from $B_R$ to the unit ball $B_1$ is well approximated by a minimal surface in $B_1$. In the classical case $p=m=2$ the approximation is made rigorous in three main steps:
\begin{itemize}
    \item[(1)] The $\Gamma$-convergence result established by Modica and Mortola in \cite{Modica,MM77}, see also \cite{Bou90,OS91}. One considers the blow down rescalings $u_{R}(x):=u(Rx)$ in $B_{1}$ for a sequence of radii $R \to \infty$. Then, up to a subsequence, $u_{R}(x)$ converges weakly in the BV norm and strongly in $L^{1}$ norm to a set characteristic function $$u_{\infty}=\chi_{E}-\chi_{E^{c}},$$ where $E$ is a Caccioppoli set with minimal perimeter in $B_1$.
    \item[(2)] The density estimate obtained by Caffarelli and C\'ordoba in \cite{CC95}, see also \cite{DFV18,FV08,PV05b,PV05a,V04}. It asserts that
    \begin{equation}
        \Big|\{u\geq0\}\cap B_{R}\Big|\geq\delta R^{n}\mbox{ and }\Big|\{u\leq0\}\cap B_{R}\Big|\geq\delta R^{n},
    \end{equation}
    implying that $0\in\partial E$ and that the convergence of the level sets of $u_R$ to $\partial E$ is uniform on compact sets of $B_1$. 
    
   When $n\leq7$, the Bernstein theorem \cite{Simonscone} implies that $\partial E$ is smooth and then the zero set of a non-trivial minimizer must be asymptotically flat at large scales. On the other hand the asymptotic flatness does not hold in general dimensions, as was shown by the example of Del Pino-Kowalczyc-Wei in \cite{DKW11}. 
    \item[(3)] The convergence of $\{u_R=0\}$ to $\partial E$ in the stronger $C^{2,\alpha}_{loc}(B_1)$ sense. This follows from the improvement of flatness result established by the first author in \cite{S09} in connection with a conjecture of De Giorgi (see also \cite{SV05,VSS06} for similar results when $p\in(1,\infty)$.)
    
\end{itemize}

As mentioned above, when $m<p$ the minimization problem produces a free boundary of Alt-Phillips type for the region $u\neq\pm1$, while when $m=p$ the convergence of $u$ to the  $\pm 1$ phases is exponential. On the other hand, the assumption $m> p$ produces less stable minimal points (still at $\pm1$) for an infinitesimal potential energy $W(v)\sim(1-v^{2})^{m}$. It is then natural to investigate whether the results above extend to these type of degenerate Ginzburg-Landau energies.

In \cite{DFV18}, Dipierro, Farina and Valdinoci considered $Q$-minimizers of such degenerate energies and obtained the density estimates for a certain range of $m$'s depending on the dimension $n$. Precisely, the authors considered general Ginzburg-Landau type energies
\begin{equation}\label{eq. GL energy}
    J(u,\Omega)=\int_{\Omega}E(\nabla u,u,x)dx
\end{equation}
where
\begin{equation}\label{eq. E(xi,tau,x) assumption}
    \lambda(|\xi|^{p}+|\tau+1|^{m}\chi_{\{\tau\leq0\}})\leq E(\xi,\tau,x)\leq\Lambda(|\xi|^{p}+|\tau+1|^{m}).
\end{equation}
We recall that the terminology $Q$-minimizer is a relaxation of the notion of minimizer up to a multiplicative factor.
\begin{definition}
    We say that $u\in W^{1,p}(\Omega,[-1,1])$ is a $Q$- minimizer of \eqref{eq. GL energy pm intro}, if for every bounded open set $\Omega' \subset\subset \Omega$, and any $v\in W^{1,p}_{loc}(\Omega,[-1,1])$ such that $u=v$ in $\Omega\setminus \Omega'$, we have
    \begin{equation}
        J(u,\Omega')\leq Q\cdot J(v,\Omega').
    \end{equation}
\end{definition}
In \cite{DFV18} it was proved that if
\begin{equation}\label{eq. power assumption, is it optimal?}
    \frac{pm}{m-p}>n
\end{equation}
and $\Big|\{u\geq0\}\cap B_{1}\Big|>c$ for some positive $c$, then
\begin{equation}
    \Big|\{u\geq0\}\cap B_{R}\Big|\geq\delta R^{n}.
\end{equation}
We remark that the density estimates in the non-degenerate case $0<m\leq p$ were obtained by Farina and Valdinoci in the earlier work \cite{FV08}. Based on these results, the density estimates hold in the following cases:
\begin{itemize}
    \item $p\geq n$ 
    \item $p<n$ and $0<m<\frac{pn}{n-p}$.
\end{itemize}

In this paper, we show that the assumption \eqref{eq. power assumption, is it optimal?} can be removed provided that there exists an initial ball $B_\rho$ of a fixed large radius in which the density estimate holds, see \eqref{eq. additional assumption} below. We state the density estimate for $Q$-minimizers.

\begin{theorem}\label{thm. main}
    Fix constants $n,p,m,Q,\lambda,\Lambda$ with $1<p<m$, and let $u$ be a $Q$-minimizer of the energy \eqref{eq. GL energy}, with $E(\xi,\tau,x)$ satisfying \eqref{eq. E(xi,tau,x) assumption}. 
    
    Given any $\epsilon >0$, there exist $\rho_{0}=\rho_{0}(\epsilon)$ large and $\delta=\delta(\epsilon)$, so that if
    \begin{equation}\label{eq. additional assumption}
        \Big|\{u\geq0\}\cap B_{\rho}\Big|\geq\epsilon\rho^{n}
    \end{equation}
    for some $\rho\geq\rho_{0}$, then
    \begin{equation}
        \Big|\{u\geq0\}\cap B_{R}\Big|\geq\delta R^{n}
    \end{equation}
    for all $R\geq\rho$ for which $B_R \subset \Omega$.
\end{theorem}

\begin{remark}
The constant $\delta$ can be taken independent of $\epsilon$ for sufficiently large $R$'s.
\end{remark}

The proof of Theorem \ref{thm. main} follows the standard outline given by Caffarelli and C\'ordoba in \cite{CC95} which was also implemented in the subsequent works  \cite{FV08,DFV18} mentioned above. The idea is to reproduce the density estimate for minimal surfaces by comparing the energy of $u$ to the energy of a radial test function $v_R$ in $B_R$ and then deduce from this a discrete ODE inequality for the volume $|\{u>0\}\cap B_R |$ for a sequence of radii $R=kT$, $k \in \mathbb N$, with $T$ a sufficiently large constant. The key improvement in our analysis comes from estimating a term in the energy inequality more carefully, by decomposing it in annular rings of width $1$. This allows us to keep track of the decay of $E(\nabla v_{R},v_{R},x)$ and then end up with a discrete ODE in which the step length is 1, instead of a large step length $T$ as in \cite{CC95,DFV18}. Equivalently, instead of writing the energy inequality for a discrete sequence of radii $R=kT$, we use it for all values of $R$ and then integrate it between intervals of length $T$. A similar technique was used in the uniform density estimates of the zero sets for minimizers of the Alt-Phillips energy with negative exponents established recently by De Silva-Savin in \cite{DS}.

Finally, we consider the case in which $u$ is a minimizer instead of a $Q$-minimizer, and the Ginzburg-Landau energy \eqref{eq. GL energy} is sufficiently smooth. By using the Euler-Lagrange equation, we verify that the assumption \eqref{eq. additional assumption} is always satisfied in a fixed neighborhood of the zero level set and then prove a general density estimate for minimizers.

\begin{theorem}\label{thm. main2}
    Let $1<p<m$, and assume that $u:\Omega \to[-1,1]$ is a minimizer of
    \begin{equation}
        J(v,\Omega)=\int_{\Omega}\Big\{|\nabla v|^{p}+(1-v^{2})^{m}\Big\}dx
    \end{equation}
    with $u(0)=0$. Then there exist some universal constants $\delta,R_{0}>0$, such that for every $R\geq R_{0}$, with $B_R \subset \Omega$, we have
    \begin{equation}
        \Big|B_{R}\cap\{u\geq0\}\Big|\geq\delta R^{n}\mbox{ and }\Big|B_{R}\cap\{u\leq0\}\Big|\geq\delta R^{n}.
    \end{equation}
\end{theorem}
\begin{remark}
    The density estimates hold for a more general class of potential energies $W(u)$, see Theorem~\ref{thm. main3}.
\end{remark}
\begin{remark}
When $p=2$, we may take $R_0=0$ as a consequence of the unique continuation property of solutions to the Laplace equation.
\end{remark}

\section{Proof of Theorem~\ref{thm. main}}
In this section we establish Theorem~\ref{thm. main} and divide the proof into six steps. Most of the computations have been carried out already in \cite{DFV18}, however for completion and clarity of the argument we provide the details. In the proof, positive constants depending only on $n,p,m,Q,\lambda,\Lambda$ are referred to as uniform constants, and they are by default denoted by $C$ or $c$, with their values changing from line to line whenever there is no possibility of confusion. Some universal constants that are recalled later in the proof are labeled as $c_0$, $C_{0}$, etc. in order to avoid circular dependence.
\subsection{The quantities $V_R$, $P_R$ and $\mathcal M_R$}
For every integer $R$, we define two increasing sequences
\begin{equation}
    V_{R}=\Big|B_{R}\cap\{u\geq0\}\Big|,\quad P_{R}=\int_{B_{R}}W(u)dx.
\end{equation}
Let $T=T(\epsilon)$ be a sufficiently large integer depending on $\epsilon$ and the universal constants to be made precise later, and for $R \ge T$ we define a weighted mixture of $V_{R}$ and $P_{R}$ as
\begin{equation}\label{eq. M_R}
    \mathcal{M}_{R}:= \sum_{j=0}^{T}\Big(P_{R-j}+\frac{V_{R-j}}{(1+j)^{\frac{pm}{m-p}}}\Big).
\end{equation}
The goal is to deduce a discrete differential inequality for the sequence $\mathcal M_R$.

We first show that
\begin{equation}\label{P_R upper bound}
    P_{R}\leq C \, R^{n-1},\quad\mbox{for all }R\geq 2.
\end{equation}
For this we consider the competitor
\begin{equation}
    v(x)=\left\{\begin{aligned}
        &-1,&\mbox{if }&|x|\leq R\\
        &|x|-R-1,&\mbox{if }&R\leq|x|\leq R+2\\
        &1,&\mbox{if }&|x|\geq R+2
    \end{aligned}\right.
\end{equation}
and notice that $J(v,B_{R+2})\leq CR^{n-1}$. We set 
$$\overline {\Omega}:=\{x:u(x)\geq v(x)\} \cap \overline{B}_{R+2}, \quad \mbox{and notice that} \quad  B_{R}\subseteq \overline{\Omega}.$$
Then, \eqref{P_R upper bound} follows by the $Q$-minimality of $u$,
$$P_{R}\leq J(u, \overline \Omega)\leq Q\cdot J(v, \overline \Omega)\leq Q \cdot J(v,B_{R+2})\leq C_1 \, R^{n-1}.$$
Since 
$$ \frac{pm}{m-p} >p >1,$$
the series 
$$ \sum_{j=0}^{\infty}(1+j)^{-\frac{pm}{m-p}} < \infty, $$
converges and is bounded above by a universal constant. Using this and \eqref{P_R upper bound} we have
$$\mathcal M_R \le (T+1) C R^{n-1} + C V_R$$ which implies a lower bound for $V_R$ in terms of $\mathcal M_R$:
\begin{equation}\label{eq. AR vs VR}
    V_{R}\geq c _0 \left( \mathcal M_R - C_0 T R^{n-1} \right).
\end{equation}

\subsection{An one-dimensional function $U(t)$} We construct a suitable one-dimensional function $U(t)$ so that
\begin{equation}\label{eq. requirement of U(t)}
    (U')^{p}\sim(1+U)^{m}\mbox{ when }t\leq0.
\end{equation}
Precisely we let
\begin{equation}\label{eq. U(t) definition}
    U(t)=\left\{\begin{aligned}
        &t&\mbox{if }&t\geq0,\\
        &-1+(1-t)^{-\frac{p}{m-p}}&\mbox{if }&t\leq0.
    \end{aligned}\right.
\end{equation}
It is then easy to find that for $t\leq0$,
\begin{equation}
    |U'(t)|\leq C(1-t)^{-\frac{m}{m-p}},\quad W\Big(U(t)\Big)\leq C(1-t)^{-\frac{mp}{m-p}}
\end{equation}
Therefore, for $T\geq1$, we have
\begin{equation}\label{E-T}
    \mathcal{E}(-T):=\int_{-\infty}^{-T}|\nabla U|^{p}+W(U)dt\leq CT^{1-\frac{pm}{m-p}}=CT^{-\gamma},
\end{equation}
with $\gamma>0$ universal.

\subsection{An energy competitor $v_{R}(x)$}
Let $R$ be an integer and we consider a competitor $v_{R}(x)$ using $U(t)$ constructed in \eqref{eq. U(t) definition}
\begin{equation}\label{eq. v_R(x) definition}
    v_{R}(x)=U(|x|-R),\quad x\in B_{R+1}.
\end{equation}
In the annulus $B_{R+1-j} \setminus B_{R-j}$ we have
\begin{align}\label{E0}
 E(\nabla v_R, v_R,x) & \le C(|\nabla v_R|^p + W(v_R)) \, \le C' \, |\nabla v_R| ^p \\
 & \le C |U'(-j)|^p \le C (1+j)^{-\frac{mp}{m-p}}.
 \end{align}
For a sufficiently large $T$, we see by \eqref{E-T}
\begin{align}\label{eq. J(v) outside}
    J(v_{R},B_{R-T})\leq&C R^{n-1}\int_{-\infty}^{-T}|\nabla U(s)|^{p}+W\Big(U(s)\Big)ds\\
    \leq&C R^{n-1}\mathcal{E}(-T)\leq CR^{n-1}T^{-\gamma}.
\end{align}

For each $h\leq1$, we define
\begin{equation}
    S_{R,h}:=\{x\in B_{R+1}:u(x)>h>v_{R}(x)\}.
\end{equation}
We also define $S$ as
\begin{equation}
    S_{R}:=\{x\in B_{R+1}:u(x)>v_{R}(x)\}=\bigcup_{h<1}S_{R,h}.
\end{equation}
Since $v_{R}(x)=1\geq u(x)$ on $\partial B_{R+1}$, it then follows that $\overline{S_{R}}\subseteq B_{R+1}$, and the inclusion is satified by all $S_{R,h}$'s.

The boundary of each $S_{R,h}$ consists of two parts, namely
\begin{equation}
    \partial S_{R,h}=\partial^{u}S_{R,h}\cup\partial^{v}S_{R,h}
\end{equation}
with
\begin{align}
    \partial^{u}S_{R,h}=&\partial S_{R,h}\cap\{u=h\}=S_{R}\cap\{u=h\},\\
    \partial^{v}S_{R,h}=&\partial S_{R,h}\cap\{v_{R}=h\}=S_{R}\cap\{v_{R}=h\}.
\end{align}
By the isoperimetric inequality, we have
\begin{equation}\label{eq. isoperimetric}
    |S_{R,h}|^{\frac{n-1}{n}}\leq C\mathcal{H}^{n-1}(\partial S_{R,h})=C\Big[\mathcal{H}^{n-1}(\partial^{u}S_{R,h})+\mathcal{H}^{n-1}(\partial^{v}S_{R,h})\Big]
\end{equation}

By Young's inequality, we see at $x\in\partial^{u}S_{R,h}$,
\begin{equation}
    W(h)^{\frac{p-1}{p}}\leq C\Big(|\nabla u(x)|^{p-1}+\frac{W(h)}{|\nabla u(x)|}\Big),
\end{equation}
while for $x\in\partial^{v}S_{R,h}$,
\begin{equation}
    W(h)^{\frac{p-1}{p}}\leq C\Big(|\nabla v(x)|^{p-1}+\frac{W(h)}{|\nabla v(x)|}\Big).
\end{equation}
We then integrate \eqref{eq. isoperimetric} with the weight $W(h)^{\frac{p-1}{p}}$, and use the co-area formula:
\begin{align*}
    &\int_{U(-T)}^{1}|S_{R,h}|^{\frac{n-1}{n}}W(h)^{\frac{p-1}{p}}dh\\
    \leq&C\int_{U(-T)}^{1}\int_{\partial^{u}S_{R,h}\cup\partial^{v}S_{R,h}}W(h)^{\frac{p-1}{p}}d\sigma dh\\
    \leq&C\Big\{\int_{U(-T)}^{1}\int_{\partial^{u}S_{R,h}}\big(|\nabla u|^{p}+W(u)\big)\frac{d\sigma}{|\nabla u(x)|}dh\\
    &+\int_{U(-T)}^{1}\int_{\partial^{v}S_{R,h}}\big(|\nabla v|^{p}+W(v)\big)\frac{d\sigma}{|\nabla v(x)|}dh\Big\}\\
    \leq&C\Big\{\int_{S_{R}}\Big(|\nabla u|^{p}+W(u)\Big)dx+\int_{S_{R}}\Big(|\nabla v|^{p}+W(v)\Big)dx\Big\}\\
    \leq&C\Big\{J(u,S_{R})+J(v_{R},S_{R})\Big\}\leq C J(v_{R},S_{R}).
\end{align*}
The last inequality is due to $u$'s being a $Q$-minimizer. In conclusion
\begin{equation}\label{eq. LHS and RHS}
\int_{U(-T)}^{1}|S_{R,h}|^{\frac{n-1}{n}}W(h)^{\frac{p-1}{p}}dh \le C J(v_{R},S_{R}),
\end{equation}
from which we will deduce the desired discrete inequality for $\mathcal M_R$.

\subsection{Lower bound for the left side of \eqref{eq. LHS and RHS}}
For any $T\geq1$ we have,
\begin{equation}
    \int_{U(-T)}^{0}W(h)^{\frac{p-1}{p}}dh\geq c,
\end{equation}
with $c$ universal. For all $h$ satisfying $U(-T)<h<0$, we see
\begin{equation}
    S_{R,h}=\{x:u(x)>h>v_{R}(x)\}\supseteq\{x:u(x)\geq0>U(-T)\geq v_{R}(x)\}.
\end{equation}
Notice that $U(-T)\geq v_{R}(x)$ implies $|x|\leq R-T$, so
\begin{equation}
    |S_{R,h}|\geq|\{u(x)\geq0\}\cap B_{R-T}|=V_{R-T}.
\end{equation}
Therefore,
\begin{equation}
    \int_{U(-T)}^{1}|S_{R,h}|^{\frac{n-1}{n}}W(h)^{\frac{p-1}{p}}dh\geq c(V_{R-T})^{\frac{n-1}{n}},
\end{equation}
which together with \eqref{eq. LHS and RHS} gives
\begin{equation}\label{eq. main}
(V_{R-T})^{\frac{n-1}{n}} \le C J(v_{R},S_{R}).
\end{equation}

\subsection{Upper bound for $J(v_{R},S_{R})$}
Next we estimate the righthand side of \eqref{eq. main} from above. We divide $S_{R}$ into the part in $B_{R-T}$ and $B_{R+1}\setminus B_{R-T}$. In $S_{R}\cap B_{R-T}$, we use the estimate \eqref{eq. J(v) outside} and have
\begin{equation}
    J(v_{R},S_{R}\cap B_{R-T})\leq J(v_{R},B_{R-T})\leq C \, R^{n-1}T^{-\gamma}.
\end{equation}
Next, we divide $B_{R+1}\setminus B_{R-T}$ into several annuli with width $1$:
\begin{equation}
    B_{R+1}\setminus B_{R-T}=\bigcup_{j=0}^{T}(B_{R+1-j}\setminus B_{R-j}).
\end{equation}
We bound $E(\nabla v_{R},v_{R},x)$ in the annular region $(B_{R+1-j}\setminus B_{R-j}) \cap S_R$ according to whether or not $u(x) \ge 0$. If $u(x) \ge 0$ then we use \eqref{E0}, 
\begin{equation}
    E(\nabla v_{R},v_{R},x)\leq C (1+j)^{-\frac{pm}{m-p}},
\end{equation}
while if $u(x)\leq0$, we use that $v_{R}\leq u $ in $S_{R}$, and obtain
$$E(\nabla v_R,v_R,x) \le C W(v_R) \le C W(u).$$
Then
\begin{align}
    J\Big(v_{R},S_{R}\cap(B_{R+j+1}\setminus B_{R+j})\Big)\leq&C\int_{B_{R+j+1}\setminus B_{R+j}}\Big\{W(u)\chi_{\{u\leq0\}}+(1+j)^{-\frac{pm}{m-p}} \chi_{\{u\geq0\}}\Big\}dx\\
    \leq&C \left[(P_{R+1-j}-P_{R-j})+ (V_{R+1-j}-V_{R-j})(1+j)^{-\frac{pm}{m-p}}\right].
\end{align}
Summing up all terms we find
\begin{equation}\label{eq. how to decide C1}
    J(v_{R},S_{R})\leq C \left( R^{n-1}T^{-\gamma}+\sum_{j=0}^{T} \left[ (P_{R+1-j}-P_{R-j})+ (V_{R+1-j}-V_{R-j})(1+j)^{-\frac{pm}{m-p}} \right] \right).
\end{equation}
and recalling definition \eqref{eq. M_R}, this can be rewritten as
\begin{equation}\label{eq. J(v_R,S_R) upper bound}
    J(v_{R},S_{R})\leq C \left( R^{n-1}T^{-\gamma}+ \mathcal M_{R+1} - \mathcal M_R \right).
\end{equation}

\subsection{The discrete inequality}
We combine \eqref{eq. J(v_R,S_R) upper bound}, \eqref{eq. main} and \eqref{eq. AR vs VR} and find the desired discrete inequality
\begin{equation}\label{discrete}
c _1 \left[(\mathcal M_{R-T} - C_0 T R^{n-1}) ^+ \right]^\frac{n-1}{n} \le R^{n-1}T^{-\gamma}+ \mathcal M_{R+1} - \mathcal M_R.
\end{equation}
Here $\delta>0$, and  $c_1$ small, $C_0$ large denote universal constants and \eqref{discrete} holds for all integers $R \ge T \ge 1$. 

 It follows that if $\sigma \le \sigma_0$ with $\sigma_0$ small universal and if 
 \begin{equation}\label{ind}
 \mathcal M_r \ge \sigma r^n, \quad \mbox{for} \quad r \in \{R, R-1,...,R-T\},
 \end{equation}
 for some $T=T(\sigma)$ large and some $R \ge \rho_1=\rho_1(\sigma,T)$ large, then \eqref{discrete} implies that \eqref{ind} continues to hold for $r=R+1$, and therefore for all $r \ge R$ by induction.

Indeed, we choose $$\rho_1 = 2^{n+1} C_0 T \sigma^{-1},$$ 
which implies that $R-T \ge R/2$, hence
$$ \mathcal M_{R-T} - C_0 T R^{n-1} \ge \sigma \left(\frac R2 \right)^n- C_0 T R^{n-1} \ge \frac \sigma2  \left( \frac R 2 \right)^n.$$
Then, by \eqref{discrete}, 
\begin{align}
 \mathcal M_{R+1} - \mathcal M_R & \ge c_2 \sigma^\frac{n-1}{n} R^{n-1} - T^{-\gamma} R^{n-1} \\
 & \ge C_1(n) \sigma R^{n-1} \\
 & \ge \sigma (R+1)^n - \sigma R^n,
  \end{align} 
which gives the claim. Here, in the second inequality, we have used that $T=T(\sigma)$ is sufficiently large, and $\sigma_0$ sufficiently small. Notice that \eqref{ind} implies the desired lower bound for $V_R$ since by \eqref{eq. AR vs VR} 
$$ V_R \ge c_0 (\mathcal M_R - C_0 T R^{n-1}) \ge c_0 \frac \sigma 2 R^n.$$
We remark that the computation above shows that as we let $r \to \infty$, the constant $\sigma$ in \eqref{ind} can be increased  up to $\sigma_0$ which is universal.

Now the conclusion of Theorem \ref{thm. main} is straightforward by choosing $$\sigma = 2^{-n} \epsilon, \quad \quad \rho_0 = \frac 12 \rho_1(\sigma,T) \ge 2T,$$ since the induction hypotheses in \eqref{ind} are satisfied:
$$\mathcal M_r \ge V_r \ge V_\rho \ge \epsilon \rho^n \ge \sigma r^n, \quad \quad \quad \mbox{if} \quad 2 \rho \ge r \ge \rho.$$

\section{Proof of Theorem~\ref{thm. main2}}
In this section, we prove Theorem~\ref{thm. main2}. We restate it below for general potential energies $W(u)$.
\begin{theorem}\label{thm. main3}
    Let $u:\Omega \to[-1,1]$ be a minimizer of \eqref{eq. GL energy pm intro}, where $W(v)$ satisfies that:
    \begin{itemize}
        \item $\lambda\leq\frac{W(\tau)}{(1-\tau^{2})^{m}}\leq\Lambda$ for every $\tau\in(-1,1)$;
        \item $W$ is $C^{1}$ and $\lambda\leq\frac{-W'(\tau)}{(1-\tau^{2})^{m-1}\cdot\mathrm{sgn}(\tau)}\leq\Lambda$ for all $\tau\in(-1,-\frac{1}{2})\cup(\frac{1}{2},1)$.
    \end{itemize}
   Assume that $u(0)=0$. Then there exist some universal constants $\delta,R_{0}>0$, such that for every $R\geq R_{0}$,
    \begin{equation}\label{eq. final goal}
        \Big|B_{R}\cap\{u\geq0\}\Big|\geq\delta R^{n}.
    \end{equation}
\end{theorem}

As mentioned in the Introduction, it suffices to verify the assumption \eqref{eq. additional assumption} in Theorem \ref{thm. main} is satisfied in fixed ball close to the origin. In fact we will show that $\{u\ge 0\}$ contains an entire ball of sufficiently large radius, and then we can apply Theorem \ref{thm. main} with $\epsilon= \omega_n$ in the larger concentric balls. Towards this aim we first prove that $u$ can be arbitrarily close to $1$ for some $x\in B_{R}$ for a sufficiently large $R$. Then by the standard Harnack principle, $u\geq0$ in a sufficiently large ball around $x$.
\begin{lemma}\label{lem. u quickly tends to 1}
    Let $u$ be a minimizer with $u(0)=0$. Given any $h<1$, then there exists some $\rho=\rho(h)$ such that
    \begin{equation}
        \max_{B_{\rho}}u\geq1-h.
    \end{equation}
\end{lemma}
\begin{proof}
    We first construct a radial supersolution $v$ in a ball $B_R$ with the property that $v \ge 1-h$ on $\partial B_R$ and $v(0) <0$. 
    
    For this we consider a one dimensional profile $U(t)$ (different from the one in the previous section) as an ODE solution to
    \begin{equation}\label{eq. ODE with epsilon error}
        p\frac{d}{dt}\Big((U')^{p-1}\Big)=W'(U)-\epsilon, \quad \quad U(0)=0,
    \end{equation}
    for some $\epsilon>0$ small, in an interval $[a,b]$ where $U$ is increasing. By multiplying with $U'$ on both sides and integrating, we obtain
    \begin{equation}
    (p-1)(U')^p-W(U)+ \epsilon U = \eta,
    \end{equation}
    for some constant $\eta$. Then we solve for $U$ as $U=G^{-1}$, with $G$ defined as
    $$G(s) = \int_0^s  \Big(\frac{W(\tau)-\epsilon \tau+\eta}{p-1}\Big)^{-\frac{1}{p}}d\tau,$$
    and the domain of definition of $G$ is the interval $$[s_0, s_1]:=\{\tau | \quad W(\tau)-\epsilon \tau+\eta \ge 0\}.$$
    In view of the properties of $W$, we can choose $\epsilon$ and $\eta$ small, depending on $h$, so that $$[0, 1- h] \subset [s_0,s_1],$$
    and so that the graphs of $W(\tau)$ and the line $\epsilon \tau - \eta$ intersect transversally at $s_0$ and $s_1$. This means that $a=G(s_0)$ and $b=G(s_1)$ are finite, and $U'(a)=U'(b)=0$.
    
  We define $v$ as the rotation of $U$ around the axis $t=-r$, with $r$ large, i.e.   
  \begin{equation}
        v(x)=\left\{\begin{aligned}
            U(a),& \quad \mbox{ if }|x|-r\leq a,\\
            U(|x|-r),& \quad \mbox{ if } a \leq|x|-r\leq b,
        \end{aligned}\right.
    \end{equation}
and check that it has the desired properties in $B_R$ with $R=r+b$. 

Indeed, $v \in  C^1$, and $p\Delta_{p}v=0<W'(v)$ in $B_{r-a}$, while in the annular region $B_{r+b} \setminus B_{r-a}$ we have
    \begin{align}
        p\Delta_{p}v-W'(v)=&p\frac{d}{dt}\Big((U')^{p-1}\Big)+\frac{p(n-1)}{r+t}(U')^{p-1}-W'(U)\\
        =&-\epsilon+\frac{C}{r+t}(U')^{p-1}<0\quad\mbox{for all }t\in[a,b],
    \end{align}
provided that $r$ is chosen sufficiently large.

    Next assume by contradiction that the conclusion of Lemma~\ref{lem. u quickly tends to 1} is not true,
        \begin{equation}
        u\leq1-h\mbox{ everywhere in }B_{\rho},
    \end{equation}
    for some $\rho \ge 2 R$ large to be specified later. Then we claim that for every point $x_{0}\in B_{\rho/2}$,
    \begin{equation}\label{eq. u is somewhere above v}
        u(x)\geq v(0),\quad\mbox{for some }x\in B_{R}(x_{0}).
    \end{equation}
  Otherwise,
    \begin{equation}
        u(x)\leq v(0) \le v(x-x_{0})\mbox{ everywhere in }B_{R}(x_{0}),
    \end{equation}
    and then we slide continuously the graph of the translation of $v(x)$ from $x_{0}$ back to the origin. Since
    \begin{equation}
        u(x)\leq1-h,
    \end{equation}
    which is smaller than the boundary value of $v(x)$, and since $p\Delta_{p}v(x)<W'(v(x))$, we see from the maximal principle that $u(x)$ stays below the graph of the translations of $v(x)$ during the sliding process. This contradicts the assumption $u(0)=0 > v(0)$.

    From \eqref{eq. u is somewhere above v}, $u(x) \in [v(0),1-h]$, and using the gradient estimates for $u$ we conclude
    \begin{equation}
        \Big|B_{R}(x_{0})\cap\{v(0)<u<1-h\}\Big|\geq c_{0}(h).
    \end{equation}
   for some small constant $c_0(h)$, that depends on $h$. Thus
    \begin{equation}\label{eq. one estimate in a grid}
        J\Big(u,B_R(x_{0})\Big)\geq c_{1}(h)\quad\mbox{for every }x_{0}\in B_{\rho/2},
    \end{equation}
  which gives
    \begin{equation}
        J(u,B_{\rho})\geq c_2(h) (\rho/R)^n.
    \end{equation}
    This violates the energy estimate $J(u,B_{\rho})\leq C \rho^{n-1}$ proved in \eqref{P_R upper bound}, provided that $\rho$ is chosen sufficiently large. Therefore, we have proved Lemma~\ref{lem. u quickly tends to 1}.
\end{proof}

\begin{proof}[Proof of Theorem~\ref{thm. main3}]
    By Lemma~\ref{lem. u quickly tends to 1}, $u(x_{h})\geq1-h$ for some $|x_{h}|\leq \rho(h)$. Notice that
    \begin{equation}
        p\Delta_{p}(1-u)=-W'(u)=O(|1-u|^{m-1})=o(|1-u|^{p-1}),
    \end{equation}
    so when $h$ is sufficiently small, there exists some $\rho_{1}(h)$ with
    \begin{equation}
        \lim_{h\to0}\rho_{1}(h)=+\infty,
    \end{equation}
    such that the Harnack principle implies that
    \begin{equation}
        u\geq0\mbox{ everywhere in }B_{\rho_{1}(h)}(x_{h}).
    \end{equation}
    We then apply Theorem~\ref{thm. main} to $u$, but with the origin moved to $x_{0}$. It follows that if $h$ is chosen universally small satisfying
    \begin{equation}
        \rho_{1}(h)\geq\rho_{0}(\omega_{n})
    \end{equation}
    for $\rho_{0}(\epsilon)$ obtained in Theorem~\ref{thm. main}, then for every
    \begin{equation}
        R\geq\rho_{1}(h),
    \end{equation}
    we have
    \begin{equation}
        \Big|B_{R}(x_{h})\cap\{u\geq0\}\Big|\geq\delta_{1}R^{n}.
    \end{equation}
    Since $|x_{h}|\leq \rho(h)$, we let
    \begin{equation}
        R_{0}=\rho(h)+\rho_{1}(h),
    \end{equation}
    which becomes universal as our choice of $h$ is universal, then for some universal $\delta$, we have
    \begin{equation}
        \Big|B_{R}\cap\{u\geq0\}\Big|\geq\delta R^{n},\quad\mbox{for all }R\geq R_{0}.
    \end{equation}
\end{proof}


\end{document}